\documentclass[11pt]{amsart}
\usepackage{graphicx,subcaption}
\usepackage{amsmath} 
\usepackage{amsthm,amsfonts,amssymb,mathrsfs,amscd,amstext,amsbsy} 
\usepackage{epic,eepic} 
\usepackage{yfonts}
\usepackage{paralist,enumerate}
\usepackage[all]{xy}
\usepackage{hyperref}
\hypersetup{colorlinks}
\usepackage{tikz}
\usetikzlibrary {positioning}
\definecolor {processblue}{cmyk}{0.96,0,0,0}
\usetikzlibrary{arrows}
\tikzset{    vertex/.style={circle,draw,minimum size=1.5em},    edge/.style={->,> = latex'}}

\newtheorem{theorem}{Theorem}[section]
\newtheorem{cor}[theorem]{Corollary}
\newtheorem{lemma}[theorem]{Lemma}
\newtheorem{conjecture}[theorem]{Conjecture}

\newtheorem{theo}[theorem]{Theorem}
\newtheorem{lem}[theorem]{Lemma}
\newtheorem{pro}[theorem]{Proposition}
\newtheorem{rem}[theorem]{Remark}
\newtheorem{exa}[theorem]{Example}
\newtheorem{que}[theorem]{Question}
\newtheorem{Definition}[theorem]{Definition}
\newtheorem*{Definition*}{Definition}

\def\qed{\hfill \ifhmode\unskip\nobreak\fi\quad\ifmmode\Box\else$\Box$\fi\\ }

\begin{document}

\title[Almost complex torus manifolds]{Almost complex torus manifolds - graphs, Hirzebruch genera, and problem of Petrie type}
\author{Donghoon Jang}
\thanks{Donghoon Jang was supported by the National Research Foundation of Korea(NRF) grant funded by the Korea government(MSIT) (2021R1C1C1004158).}
\address{Department of Mathematics, Pusan National University, Pusan, South Korea}
\email{donghoonjang@pusan.ac.kr}

\begin{abstract}
Let a $k$-dimensional torus $T^k$ act on a $2n$-dimensional compact connected almost complex manifold $M$ with isolated fixed points. As for circle actions, we show that there exists a (directed labeled) multigraph that encodes weights at the fixed points of $M$. This includes the notion of a GKM graph as a special case that weights at each fixed point are pairwise linearly independent. If in addition $k=n$, i.e., $M$ is an almost complex torus manifold, the multigraph is a graph; it has no multiple edges. 

We show that the Hirzebruch $\chi_y$-genus $\chi_y(M)=\sum_{i=0}^n a_i(M) \cdot (-y)^i$ of an almost complex torus manifold $M$ satisfies $a_i(M) > 0$ for $0 \leq i \leq n$. In particular, the Todd genus of $M$ is positive and there are at least $n+1$ fixed points.

Petrie's conjecture asserts that if a homotopy $\mathbb{CP}^n$ admits a non-trivial circle action, its Pontryagin class agrees with that of $\mathbb{CP}^n$. Petrie proved this conjecture if instead it admits a $T^n$-action. We prove that if a $2n$-dimensional almost complex torus manifold $M$ only shares the Euler number with the complex projective space $\mathbb{CP}^n$, an associated graph agrees with that of a linear $T^n$-action on $\mathbb{CP}^n$; consequently $M$ has the same weights at the fixed points, Chern numbers, equivariant cobordism class, Hirzebruch $\chi_y$-genus, Todd genus, and signature as $\mathbb{CP}^n$. If furthermore $M$ is equivariantly formal, the equivariant cohomology and the Chern classes of $M$ and $\mathbb{CP}^n$ also agree.
\end{abstract}

\maketitle

\tableofcontents

\section{Introduction and statements of main results}

$\indent$

An \textbf{almost complex manifold} is a manifold $M$ with a bundle map $J:TM \to TM$, called an \textbf{almost complex structure}, which restricts to a linear complex structure on each tangent space. A \textbf{unitary manifold} (also called a \textbf{stable complex manifold} or a \textbf{weakly almost complex manifold}) is a manifold $M$ with a complex structure $J$, called a \textbf{unitary structure} (also called a \textbf{stable complex structure} or a \textbf{weakly almost complex structure}), on the vector bundle $TM \oplus \underline{\mathbb{R}}^k$ for some non-negative integer $k$, where $\underline{\mathbb{R}}^k$ denotes the $k$-dimensional trivial bundle over $M$. Every complex manifold is almost complex, and every symplectic manifold admits a compatible almost complex structure.

Let $M$ be a $2n$-dimensional unitary manifold. Let a $k$-dimensional torus $T^k$ act on $M$. Throughout this paper, any group action on a unitary (almost complex) manifold is assumed to preserve the unitary (almost complex) structure, that is, $dg \circ J=J \circ dg$ for all elements $g$ of the group. Let $F$ be a fixed component of the action. Let $\dim F=2m$ and let $p$ be a point in $F$. The normal space $N_pF$ of $N$ at $p$ decomposes into the sum of $n-m$ complex $1$-dimensional vector spaces $L_1$, $\cdots$, $L_{n-m}$, where on each $L_i$ the torus $T^k$ acts by multiplication by $g^{w_{p,i}}$ for all $g \in T^k$, for some non-zero element $w_{p,i}$ of $\mathbb{Z}^k$, $1 \leq i \leq n-m$. These elements $w_{p,1}$, $\cdots$, $w_{p,n-m}$ are the same for all $p \in F$ and called \textbf{weights} of $F$.

For a circle action on a compact almost complex manifold $M$ with isolated fixed points, there exists a multigraph which encodes weights at the fixed points of $M$ (see for instance \cite{GS}, \cite{JT}), and it is useful to consider such a multigraph to study an almost complex $S^1$-manifold with isolated fixed points. In this paper, we show that as for circle actions, for a torus action there also is a multigraph encoding the fixed point data of a manifold. We first generalize the notion of a multigraph for a circle action to that for a torus action.

\begin{Definition} \label{d11}
A \textbf{labeled directed $k$-multigraph} $\Gamma$ is a set $V$ of vertices, a set $E$ of
edges, maps $i : E \to V$ and $t : E \to V$ giving the initial and terminal vertices of
each edge, and a map $w:E \to \mathbb{Z}^k$ ($\mathbb{N}^+$ if $k=1$) giving the label of each edge.
\end{Definition}

Let $w=(w_1,\cdots,w_k)$ be an element in $\mathbb{Z}^k$. By $\ker w$ we shall mean the subgroup of $T^k$ whose elements fix $w$. That is,
\begin{center}
$\ker w=\{g=(g_1,\cdots,g_k) \in T^k \subset \mathbb{C}^k \, | \, g^w:=g_1^{w_1} \cdots g_k^{w_k}=1\}$.
\end{center}
For an action of a group $G$ on a manifold $M$, we denote its fixed point set by $M^G$, i.e., $M^G=\{m \in M\, | \, g \cdot m=m, \forall g \in G\}$.

\begin{Definition} \label{d12}
Let a $k$-dimensional torus $T^k$ act on a compact almost complex manifold $M$ with isolated fixed points. We say that a (labeled directed $k$-)multigraph $\Gamma=(V,E)$ \textbf{describes} $M$ if the following hold:
\begin{enumerate}[(i)]
\item The vertex set $V$ is equal to the fixed point set $M^{T^k}$.
\item The multiset of the weights at $p$ is $\{w(e) \, | \, i(e)=p\} \cup \{-w(e) \, | \, t(e)=p\}$ for all $p \in M^{T^k}$.
\item For each edge $e$, the two endpoints $i(e)$ and $t(e)$ are in the same component of the isotropy submanifold $M^{\ker w(e)}$.
\end{enumerate}
\end{Definition}

Definition \ref{d12} means that for an edge $e$, a fixed point $i(e)$ has weight $w(e)$ and a fixed point $t(e)$ has weight $-w(e)$.

\begin{pro} (Proposition \ref{p22}) \label{p13}
Let a $k$-dimensional torus $T^k$ act on a compact almost complex manifold $M$ with isolated fixed points. There exists a (labeled directed $k$-)multigraph $\Gamma$ describing $M$ that has no self-loops. 
\end{pro}

For an action of a $k$-dimensional torus $T^k$ on a manifold $M$, the \textbf{equivariant cohomology} of $M$ is
\begin{center}
$\displaystyle H_{T^k}^{*}(M)=H^{*}(M \times_{T^k} ET^k)$, 
\end{center}
where $ET^k$ is a contractible space on which $T^k$ acts freely. 

Let $M$ be a compact almost complex manifold equipped with an action of a torus $T^k$ that has isolated fixed points. If weights at each fixed point are pairwise linearly independent, such a manifold is called an (almost complex) GKM manifold, and a labeled multigraph so called a GKM graph is associated \cite{GKM}. The GKM graph enables us to compute the (equivariant) cohomology of a given GKM manifold if furthermore the action is equivariantly formal; the map $\iota^* :H_{T^k}^*(M) \to H^*(M)$ is surjective. Some papers include equivariant formality in the definition of a GKM manifold (e.g. \cite{GW}, \cite{Ku}), while the others add the assumption separately (e.g. \cite{GKZ}, \cite{GZ}). We note that if $M$ satisfies the GKM condition that weights at each fixed point are pairwise linearly independent, our multigraph is a GKM graph (forgetting the direction of each edge), and hence our notion of a multigraph generalizes the notion of a GKM graph to a general torus action with isolated fixed points; see Corollary \ref{c24}.

An \textbf{almost complex torus manifold} is a $2n$-dimensional compact connected almost complex manifold equipped with an effective $T^n$-action that has fixed points. Because the dimension of the torus is the half of the dimension of a manifold and there are fixed points, there are only finitely many fixed points. For an almost complex torus manifold, we show that a multigraph describing it is a graph; it has no multiple edges.

\begin{pro} (Proposition \ref{p28}) \label{p14}
Let $M$ be an almost complex torus manifold. Then there is a graph describing $M$ that has no self-loops. In other words, there is a multigraph describing $M$ that has no multiple edges and no self-loops.
\end{pro}

Geometrically, this means that given a $2n$-dimensional almost complex torus manifold, at each fixed point $p$ there are $n$ $2$-spheres equipped with rotations (restrictions of the $T^n$-action on the manifold to these $2$-spheres), all of which have $p$ as one common fixed point, but no two of them share other fixed point; see Proposition \ref{p29}. Proposition \ref{p14} implies that a (multi)graph describing a $2n$-dimensional almost complex torus manifold has at least $n+1$ vertices, hence we obtain a lower bound on the number of fixed points. 

\begin{cor} \label{c16}
Let $M$ be an almost complex torus manifold. Then there are at least $\frac{1}{2} \dim M+1$ fixed points.
\end{cor}

The bound in Corollary \ref{c16} is sharp; an effective linear $T^n$-action on $\mathbb{CP}^n$ has $n+1$ fixed points; see Example \ref{e111}. If $T^n$ acts effectivaly on a $2n$-dimensional compact unitary manifold $M$ and $M$ does not bound a unitary manifold equivariantly, Darby proved that there are at least $n+1$ fixed points, using equivariant complex bordism theory \cite{D}.

For a torus action on a compact manifold $M$, the Euler number of $M$ is equal to the sum of the Euler numbers of its fixed components \cite{Kob}. Since an almost complex torus manifold has only isolated fixed points and the Euler number of a point is $1$, we also obtain a lower bound on the Euler number. Masuda also obtained this lower bound by using the theory of multi-fans \cite{Ma2}.

\begin{cor} \label{c17}
Let $M$ be an almost complex torus manifold. Then its Euler number is at least $\frac{1}{2} \dim M+1$.
\end{cor}

For a compact almost complex manifold, the \textbf{Hirzebruch $\chi_y$-genus} is the genus of the power series $\frac{x(1+ye^{-x(1+y)})}{1-e^{-x(1+y)}}$, and the \textbf{Todd genus} is the genus of the power series $\frac{x}{1-e^{-x}}$; the power series for the Todd genus is obtained by taking $y=0$ in the power series for the Hirzebruch $\chi_y$-genus. Denote by $\chi_y(M)$ the Hirzebruch $\chi_y$-genus of $M$. The Hirzebruch $\chi_y$-genus of $M$ contains three topological information; the Euler number (when $y=-1$), the Todd genus (when $y=0$), and the signature (when $y=1$) of $M$.

Let $M$ be a $2n$-dimensional compact almost complex manifold. Let $\chi_y(M)=\sum_{i=0}^n a_i(M) \cdot (-y)^i$ denote the Hirzebruch $\chi_y$-genus $\chi_y(M)$ of $M$, for some integers $a_i(M)$, $0 \leq i \leq n$. Note that the standard convention is $\chi_y(M)=\sum_{i=0}^n \chi^i(M) \cdot y^i$; hence $a_i(M)=(-1)^i \chi^i(M)$. We shall use the coefficients $a_i(M)$ so that statements and proofs become clearer. In this convention, if a torus acts on $M$ with isolated fixed points, each fixed point contributes some $a_i(M)$ by $+1$. If we use $\chi^i(M)$ then there is a sign issue that each isolated fixed point contributes $\chi^i(M)$ by $(-1)^i$ for some $i$. As noted above, the Todd genus $\mathrm{Todd}(M)$ of $M$ is equal to the Hirzebruch $\chi_y$-genus of $M$ evaluated at $y=0$, that is, $\mathrm{Todd}(M)=\chi_0(M)=a_0(M)$. In \cite{Ma2}, Masuda proved that the Todd genus of an almost complex torus manifold is positive. In this paper, we prove that in fact all coefficients $a_i(M)$ of the Hirzebruch $\chi_y$-genus of an almost complex torus manifold $M$ are positive. 

\begin{theo} \label{t15} (Theorem \ref{t32})
Let $M$ be a $2n$-dimensional almost complex torus manifold. Then $a_i(M) > 0$ for $0 \leq i \leq n$, where $\chi_y(M)=\sum_{i=0}^n a_i(M) \cdot (-y)^i$ is the Hirzebruch $\chi_y$-genus of $M$. In particular, the Todd genus of $M$ is positive.
\end{theo}

The conclusion of Theorem \ref{t15} is optimal in a sense that we cannot weaken any condition. If we decrease the dimension of a torus acting on $M$ or if we weaken almost complex structure to unitary structure, the conclusion of Theorem \ref{t15} need hot hold; see the end of Section \ref{s3}.

One may ask if there exists an almost complex torus manifold with minimal Hirzebruch $\chi_y$-genus, i.e., a manifold $M$ such that $a_i(M)$ is precisely 1 for all $0 \leq i \leq n$. The complex projective space $\mathbb{CP}^n$ is such a manifold; it admits an effective action of $T^n$ with $n+1$ fixed points and hence is an (almost) complex torus manifold. The Hirzebruch $\chi_y$-genus of $\mathbb{CP}^n$ is $\chi_y(\mathbb{CP}^n)=\sum_{i=0}^n (-y)^i=1-y+y^2+\cdots+(-y)^n$ and hence $a_i(\mathbb{CP}^n)=1$ for $0 \leq i \leq n$.

Note that for any positive integer $k$ and for any integer $n>1$,  Masuda constructed a $2n$-dimensional almost complex torus manifold $M$ whose Todd genus is $k$ \cite{Ma2}. If $n=1$, then $M$ must be the 2-sphere $S^2$ and its Todd genus is 1.

For an almost complex torus manifold $M$, each fixed point $p$ contributes some $a_i(M)$ by $1$ and hence $\sum_{i=0}^n a_i(M)$ is equal to the total number of fixed points. Since $a_i(M)>0$ for $0 \leq i \leq n$, Corollary \ref{c16} also follows from Theorem \ref{t15}.

Finally, we discuss a problem of Petrie type for almost complex torus manifolds. Petrie conjectured that if a homotopy $\mathbb{CP}^n$ admits a non-trivial $S^1$-action, then it has the same Pontryagin class as $\mathbb{CP}^n$.

\begin{conjecture}[\textbf{Petrie's conjecture}] \cite{P1}
Let $M$ be a $2n$-dimensional compact oriented manifold which is homotopy equivalent to the complex projective space $\mathbb{CP}^n$. If $M$ admits a non-trivial $S^1$-action, the total Pontryagin class of $M$ agrees with that of $\mathbb{CP}^n$.
\end{conjecture}

A weak version of the conjecture adds an assumption that the action on $M$ has isolated fixed points. While the Petrie's conjecture remains open in its full generality, there are many partial or related results. Petrie proved the conjecture if a homotopy $\mathbb{CP}^n$ admits an action of the torus $T^n$ instead of an $S^1$-action \cite{P2}. Dejter confirmed the Petrie's conjecture in dimension up to $6$ \cite{De}, Musin confirmed it in dimension up to $8$ \cite{Mu}, and James confirmed in dimension 8 \cite{Jam}. Dessai and Wilking reduced the dimension of a torus acting on a manifold by showing that the conclusion of the Petrie's conjecture holds if a homotopy $\mathbb{CP}^n$ admits a $T^k$-action with $2n \leq 8k-4$ \cite{DW}.

The Petrie's conjecture can be thought of as a particular case of the following question.

\begin{que} If a manifold equipped with a group action (torus action) shares some information with $\mathbb{CP}^n$, then what other information does the manifold share with $\mathbb{CP}^n$? \end{que}

For instance, Tsukada and Washiyama \cite{TW} and Masuda \cite{Ma1} proved that the conclusion of the Petrie's conjecture holds if a compact oriented $S^1$-manifold with the same cohomology as $\mathbb{CP}^n$ has three or four fixed components, respectively. Hattori proved the Petrie's conjecture under an assumption that a compact unitary manifold has the same cohomology of $\mathbb{CP}^n$ with first Chern class $(n+1)x$ and admits an $S^1$-action \cite{H1}. Tolman asked a symplectic analogue of the Petrie's conjecture that, if a compact symplectic manifold $M$ with $H^{2i}(M;\mathbb{R}) \cong H^{2i}(\mathbb{CP}^n;\mathbb{R})$ for all $i$ admits a Hamiltonian $S^1$-action, then $H^j(M;\mathbb{Z}) \cong H^j(\mathbb{CP}^n;\mathbb{Z})$ for all $j$ \cite{T}. In the same paper Tolman answered this question affirmatively in dimension up to 6, and the work of Godinho and Sabatini \cite{GS} together with the work of Tolman and the author \cite{JT} confirmed it in dimension 8. Motivated by the Petrie's conjecture, Masuda and Suh asked if an isomorphism between the cohomology rings of two torus manifolds preserves the Pontryagin classes of them \cite{MS}.

In the Petrie's conjecture (or in the above analogues), for a $2n$-dimensional manifold with an $S^1$-action (or $T^n$-action), being homotopy equivalent to $\mathbb{CP}^n$ (or having the same comohology of $\mathbb{CP}^n$, etc) can be regarded as an assumption which is somehow strong enough to conclude that it has the same Pontryagin class as $\mathbb{CP}^n$. In this paper, we show that for an almost complex torus manifold $M$, having the same Euler number as $\mathbb{CP}^n$ is enough to force $M$ to share many other invariants with $\mathbb{CP}^n$. More precisely, we prove that if an almost complex torus manifold only has the same Euler number as $\mathbb{CP}^n$, then a graph describing it agrees with that for a linear action on $\mathbb{CP}^n$. Consequently, weights at the fixed points, all the Chern numbers, the Hirzebruch $\chi_y$-genus, the Todd genus, and the signature of $M$ and $\mathbb{CP}^n$ agree. Since $M$ and $\mathbb{CP}^n$ have the same Chern numbers, they are equivariantly cobordant. Moreover, if in addition the action on $M$ is equivariantly formal, then the rational equivariant cohomology and the Chern classes of $M$ and $\mathbb{CP}^n$ also agree.

\begin{theo} \label{t19} (Theorem \ref{t46})
Let $M$ be a $2n$-dimensional almost complex torus manifold with Euler number $n+1$. Then the following invariants of $M$ and $\mathbb{CP}^n$ are equal. Here, $\mathbb{CP}^n$ is equipped with a linear $T^n$-action.
\begin{enumerate}
\item A (multi)graph describing it.
\item The weights at the fixed points.
\item All the Chern numbers.
\item Equivariant cobordism class.
\item The Hirzebruch $\chi_y$-genus.
\item The Todd genus.
\item The signature.
\end{enumerate}
If furthermore the action on $M$ is equivariantly formal, the following invariants of $M$ and $\mathbb{CP}^n$ are also equal.
\begin{enumerate}
\item[(8)] The rational equivariant cohomology.
\item[(9)] The Chern classes.
\end{enumerate}
\end{theo}

\begin{rem} \label{r111}
Note that in Theorem \ref{t19},
\begin{itemize}
\item (1) implies (2) by Definition \ref{d12}.
\item (2) implies (3) because Chern numbers are computed in terms of the weights at the fixed points in the Atiyah-Bott-Berline-Vergne localization formula (Theorem \ref{t43}).
\item (3) implies (4) because two manifolds are equivariantly cobordant if and only if they have the same Chern numbers.
\item (3) implies (5) because the coefficients of the Hirzebruch $\chi_y$-genus can be computed as rational combinations of the Chern numbers.
\item (5) implies (6) and (7). 
\item Alternatively, having the Euler number $n+1$ implies (5); see Theorem \ref{t42}.
\item (1) implies (8) if in addition the action on $M$ is equivariantly formal, because our graph is a GKM graph, and a GKM graph determines the equivariantly cohomology of a given manifold.
\item (9) follows from (1) and Proposition 3.4 of \cite{GKZ}. Proposition 3.4 of \cite{GKZ} states that our graph describing an almost complex torus manifold (which is called a signed graph in \cite{GKZ}) determines its Chern classes if it is equivariantly formal.
\end{itemize}
\end{rem}

\begin{exa}\label{e111}
By a linear $T^n$-action on $\mathbb{CP}^n$ we mean an action
\begin{center}
$g \cdot [z_0: z_1:\cdots :z_n]=[z_0:g^{a_1} z_1: \cdots :g^{a_n}z_n]$
\end{center}
for all $g \in T^n \subset \mathbb{C}^n$, where $a_1$, $\cdots$, $a_n$ form a basis of $\mathbb{Z}^n$. Let $a_0=(0,\cdots,0) \in \mathbb{Z}^n$. The action has $n+1$ fixed points $p_0=[1:0:\cdots:0]$, $\cdots$, $p_n=[0:\cdots:0:1]$, and the weights at $p_i$ are $\{a_j-a_i\}_{j \neq i}$.

Now, we associate a (multi)graph describing this $T^n$-action on $\mathbb{CP}^n$. To each fixed point $p_i$ we assign a vertex (also denoted by $p_i$). For $i<j$, the fixed points $p_i$ and $p_j$ are in the fixed component $[0:\cdots:0:z_i:0:\cdots:0:z_j:0:\cdots:0]$ of $M^{\ker (a_j-a_i)}$, which is the 2-sphere, on which the $T^n$-action on $\mathbb{CP}^n$ restricts to act by
\begin{center}
$g \cdot [0:\cdots 0:z_i:0:\cdots:0:z_j:0:\cdots:0]$

$=[0:\cdots:0:g^{a_i} z_i:0:\cdots:0:g^{a_j}z_j:0:\cdots:0]$,
\end{center}
giving $p_i$ and $p_j$ weight $a_j-a_i$ and $a_i-a_j$ for this action, respectively. Therefore, for $i<j$ we draw an edge from $p_i$ to $p_j$ and give the edge label $a_j-a_i$. Let $\Gamma$ be a (multi)graph obtained. Then $\Gamma$ describes the linear $T^n$-action on $\mathbb{CP}^n$. Figure \ref{fig1} is such a graph for $n=3$.

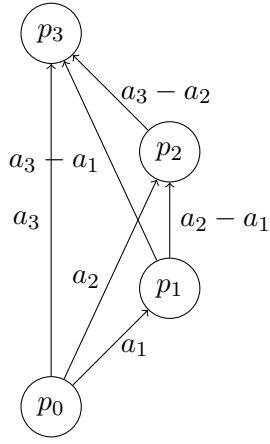
\begin{figure} 
\centering
\begin{tikzpicture}[state/.style ={circle, draw}]
\node[state] (a) {$p_{0}$};
\node[state] (b) [above right=of a] {$p_1$};
\node[state] (c) [above=of b] {$p_2$};
\node[state] (d) [above left=of c] {$p_3$};
\path (a) [->] edge node[right] {$a_1$} (b);
\path (a) [->] edge node[left] {$a_2$} (c);
\path (a) [->] edge node[left] {$a_3$} (d);
\path (b) [->] edge node [right] {$a_2-a_1$} (c);
\path (b) [->] edge node [left] {$a_3-a_1$} (d);
\path (c) [->] edge node [right] {$a_3-a_2$} (d);
\end{tikzpicture}
\caption{Graph describing a linear $T^3$-action on $\mathbb{CP}^3$}\label{fig1}
\end{figure}
\end{exa}

We describe the equivariant cohomology of an almost complex torus manifold $M$ with minimal Euler number $\frac{1}{2}\dim M+1$ which is equivariantly formal. Let $p_0$, $\cdots$, $p_n$ denote the fixed points, where $\dim M=2n$. To each fixed point $p_i$ we assign a vertex, also denoted by $p_i$. For each $i<j$ we draw an edge from $p_i$ to $p_j$ and give the edge label $w_{i,j}$. Let $\Gamma$ be a graph obtained. By Proposition \ref{p14} and Remark \ref{r25}, $\Gamma$ describes $M$. The proof of Theorem \ref{t19} shows that $w_{i,j}=w_{0,j}-w_{0,i}$ for $0<i<j$. Therefore, the graph $\Gamma$ is the same as the graph describing a linear $T^n$-action on $\mathbb{CP}^n$ in Example \ref{e111}. Because $M$ is equivariantly formal, there is a $H_{T^n}^*(\textrm{pt})$-module isomorphism $H_{T^n}^*(X)=H^*(X) \otimes H_{T^n}^*(\textrm{pt})$. Moreover, by the main theorem of \cite{GKM},
\begin{center}
$\displaystyle H_{T^n}^*(M;\mathbb{Q}) \simeq \{(f_{p_i}) \in \bigoplus_{p_i \in M^{T^n}} H^*(BT^n;\mathbb{Q}) \, | \, f_{p_i}-f_{p_0} \in (w_{0,i}) \textrm{ for }1 \leq i \leq n, f_{p_j}-f_{p_i} \in (w_{0,j}-w_{0,i})\textrm{ for } 0<i<j\}$,
\end{center}
where $(w_{0,i})$ is the ideal generated by the weight $w_{0,i} \in H^2(BT^n;\mathbb{Q})$ (and similarly for $(w_{0,j}-w_{0,i})$) and $BT^n$ is the classifying space of $T^n$. Setting $w_{0,0}$ to be the zero weight $(0,\cdots,0)$, we can simply this to
\begin{center}
$\displaystyle H_{T^n}^*(M;\mathbb{Q})$ 

$\displaystyle \simeq \{(f_{p_i}) \in \bigoplus_{p_i \in M^{T^n}} H^*(BT^n;\mathbb{Q}) \, | \, f_{p_j}-f_{p_i} \in (w_{0,j}-w_{0,i})\textrm{ for } i<j\}$.
\end{center}
Because $\Gamma$ describes both of $M$ and $\mathbb{CP}^n$ with a linear action of Example \ref{e111}, their equivariant cohomologies are the same. Also, note that the $n$-weights $w_{0,1}$, $\cdots$, $w_{0,n}$ at $p_0$ completely determines the equivariant cohomology of $M$. Because $p_0$ is arbitrary, the weights at any fixed point determine it.

The conclusion of Theorem \ref{t19} is optimal in a sense that it need not hold if we reduce the dimension of a torus acting on a manifold or we weaken almost complex structure to unitary structure. 

First, we cannot reduce the dimension of a torus. The 6-sphere $S^6$ admits an almost complex structure and admits a $T^2$-action with 2 fixed points $p$ and $q$ that have weights $\{-a-b,a,b\}$ and $\{-a,-b,a+b\}$ for some non-zero $a \neq b \in \mathbb{Z}^2$. We blow up $p$. Blowing up $p$ equivariantly, instead of $p$ we get three fixed points $p_1$, $p_2$, and $p_3$ that have weights $\{-2a-b,a,b-a\}$, $\{-a-b,2a+b,a+2b\}$, and $\{a-b,-a-2b,b\}$, respectively. The $T^2$-action on $S^6$ extends to a $T^2$-action on the blown up manifold $\widetilde{S^6}$ with 4 fixed points $p_1$, $p_2$, $p_3$, and $q$ and hence the Euler number of $\widetilde{S^6}$ is 4, but its Hirzebruch $\chi_y$-genus is $\chi_y(\widetilde{S^6})=-2y+2y^2 \neq 1-y+y^2-y^3=\chi_y(\mathbb{CP}^3)$. This was discussed for an $S^1$-action in \cite{Jan3} and naturally extends to a $T^2$-action.

Second, in Theorem \ref{t19} we cannot weaken almost complex structure to unitary structure. Consider the $2$-sphere $S^2=\{(z,x) \in \mathbb{C} \times \mathbb{R} \, : \, |z|^2+x^2=1\}$ as a unitary manifold $TS^2 \oplus \underline{\mathbb{R}}^2=S^2 \times \mathbb{C}^2$. Let $S^1$ act on $S^2$ by $g \cdot (z,x)=(gz,x)$ for all $g \in S^1 \subset \mathbb{C}$. Its Hirzebruch $\chi_y$-genus with this unitary structure vanishes, while the Hirzebruch $\chi_y$-genus of $\mathbb{CP}^1$ is $1-y$. We discuss $T^n$-actions on $S^{2n}$ as unitary manifolds more in details at the end of Section \ref{s3}.

The structure of this paper is as follows. In Section \ref{s2}, we show that for a torus action on a compact almost complex manifold with isolated fixed points there exists a multigraph describing it (Proposition \ref{p13}), and for an almost complex torus manifold a multigraph describing it is a graph (Proposition \ref{p14}). In Section \ref{s3}, we show that all the coefficients $a_i(M)$ of the Hirzebruch $\chi_y$-genus $\chi_y(M)=\sum_{i=0}^{n} a_i(M) \cdot (-y)^i$ of an almost complex torus manifold $M$ are positive (Theorem \ref{t15}) and discuss its applications. In Section \ref{s4}, we prove a Petrie type problem for almost complex torus manifolds, Theorem \ref{t19}, in the proof of which Proposition \ref{p14} and Theorem \ref{t15} play crucial roles.

\section*{Acknowledgements}

The author would like to thank Mikiya Masuda for valuable comments on an earlier version of this paper, one of which is that the theory of multi-fans in \cite{Ma2} can prove (2) Theorem \ref{t19} that if a $2n$-dimensional almost complex torus manifold $M$ has Euler number $n+1$, the multi-fan of $M$ agrees with the ordinary fan of $\mathbb{CP}^n$ up to unimodular transformation and this implies (2).

\section{Multigraphs for torus actions} \label{s2}

Let the circle act on a compact almost complex manifold $M$ with isolated fixed points. For each fixed point $p$, the \textbf{index} of $p$ is the number of negative weights at $p$. In \cite{JT}, Tolman and the author proved that a multigraph describing $M$ exists. Moreover, there exists one without any self-loops. 

\begin{lem} \cite{JT} \label{l21}
Let the circle act on a compact almost complex manifold $M$ with isolated fixed points. There exists a (directed labeled) multigraph describing $M$ such that, for each edge $e$, the index of $i(e)$ in the isotropy submanifold $M^{\mathbb{Z}/(w(e))}$ is one less than the index of $t(e)$ in $M^{\mathbb{Z}/(w(e))}$. In particular, the multigraph has no self-loops.
\end{lem}

For a torus action on a compact almost complex manifold, there also is a multigraph which encodes weights at the fixed points. For a labeled directed multigraph $\Gamma$, if $e$ is an edge of $\Gamma$ and has label $w(e)$, we say that $(i(e),t(e))$ is $w(e)$-edge, and $(t(e),i(e))$ is $-w(e)$-edge.

\begin{pro} \label{p22}
Let a $k$-dimensional torus $T^k$ act on a $2n$-dimensional compact almost complex manifold $M$ with isolated fixed points.
\begin{enumerate}[(1)]
\item For any $w \in \mathbb{Z}^k$,
\begin{center}
$\displaystyle \sum_{p \in M^{S^1}} N_p(w)=\sum_{p \in M^{S^1}} N_p(-w)$,
\end{center}
where $N_p(w)=|\{ i \, : \, w_{pi}=w, 1 \leq i \leq n\}|$ is the number of times $w$ occurs as a weight at $p$.
\item There exists a multigraph $\Gamma$ describing $M$ that has no self-loops. 
\end{enumerate}
\end{pro}

\begin{proof}
If $k=1$, (2) of this proposition is Lemma \ref{l21}, and (1) follows from (2). Therefore, from now on, we assume that $k \geq 2$.

Let $p$ be a fixed point and let $w$ be a weight at $p$. Let $F$ be a connected component of $M^{\ker w}$ which contains $p$. The component $F$ is a smaller dimensional $\ker w$-invariant compact almost complex submanifold of $M$.

Let $\xi$ be a non-zero integral element of the Lie algebra $\mathfrak{t}$ of $T^k$ not in $\ker w$ (i.e., $\xi \in \mathbb{Z}^k \cap (\mathfrak{t} \setminus \ker w')$), which generates a subcircle $S$ of $T^k$ that has the same fixed point set as the $T^k$-action on $M$, i.e., $M^S=M^{T^k}$. Such an $S$ exists because there is a finite number of orbit types. Replacing $\xi$ by $-\xi$ if necessary, we may assume that $\langle \xi,w \rangle$ is positive, where $\langle \cdot, \cdot \rangle$ denotes the usual inner product. The circle subgroup $S$ acts on $F$, and a point $q$ in $F$ is fixed by the $S$-action if and only if it is fixed by the $T^k$-action, that is, $F^S=F^{T^k}=F \cap M^{T^k}$. Moreover, if $q \in F$ is a $T^k$-fixed point, any $T^k$-weight in $T_qF$ is an integer multiple of $w$. If $w_{q,i}$ is a $T^k$-weight in $T_qF$, the corresponding $S$-weight is $\langle \xi,w_{q,i} \rangle$. 

Applying Lemma \ref{l21} to the $S$-action on $F$, there exists a labeled directed multigraph $\Gamma_F$ describing the $S$-action on $F$ that has no self-loops. Suppose that $(p_1,p_2)$ is an edge in $\Gamma_F$. Let $x=\langle \xi,w \rangle$. For $q \in F^S=F^{T^k}$, the $S$-weights in $T_qF$ are integer multiples of $x$, and hence the label of any edge in $\Gamma_F$ is an integer multiple of $x$. Suppose that $(p_1,p_2)$ is an edge in $\Gamma_F$ which has label $bx$ for some positive integer $b$. The $S$-weight $\pm bx$ only comes from the $T^k$-weight $\pm bw$. We replace the label $bx$ of the $(p_1,p_2)$-edge by $bw$, which is a weight of the $T^k$-action at $p_1$ (and $p_2$ has $T^k$-weight $-bw$). In particular, after doing this, $(p,p')$ is $w$-edge for some vertex (fixed point) $p'$ such that $p \neq p'$ and $p' \in F^{T^k}$.

Repeating this argument for all weights over all fixed points, we get a labeled directed multigraph describing $M$. This proves (2) of this proposition. Then (1) of this proposition follows from (2). \end{proof}

(1) of Proposition \ref{p22} implies the following result, which was proved for circle actions in \cite{H}.

\begin{cor} \label{c23}
Let a $k$-dimensional torus $T^k$ act on a $2n$-dimensional compact almost complex manifold with isolated fixed points. Then 
\begin{center}
$\displaystyle \sum_{p \in M^{S^1}} \sum_{i=1}^n w_{p,i}=0 \in \mathbb{Z}^k$.
\end{center}
\end{cor}

If a torus action on a compact almost complex manifold with isolated fixed points satisfies the GKM condition that weights at each fixed point are pairwise linearly independent (and the action is equivariantly formal), our multigraph is a GKM graph.

\begin{cor} \label{c24}
Let a $k$-dimensional torus $T^k$ act on a compact almost complex manifold $M$ with isolated fixed points. Suppose that weights at each fixed point are pairwise linearly independent. Then a multigraph describing $M$ is a GKM graph, forgetting the direction of each edge.
\end{cor}

\begin{proof}
Let $p$ be a fixed point and let $w$ be a weight at $p$. As in the proof of \ref{p22}, we consider a connected component $F$ of $M^{\ker w}$ which contains $p$. Since weights at $p$ are pairwise linearly independent, it follows that $F$ is $2$-dimensional. The torus action on $M$ restricts to act on a compact almost complex manifold $F$, having $p$ as a fixed point. It follows that $F$ is the $2$-sphere and has another fixed point $p$. We draw an edge $e$ from $p$ to $q$, giving a label $w$. If we repeat this argument for all weights over all fixed points, the associated multigraph is a GKM graph. \end{proof}

\begin{rem}\label{r25} Let a torus $T^k$ act on a compact almost complex manifold $M$ with isolated fixed points. Suppose that a multigraph $\Gamma$ describes $M$. Let $e$ be an edge of $\Gamma$, and let $i(e)$, $t(e)$, and $w(e)$ be the initial vertex, the terminal vertex, and the label of the edge $e$, respectively. Then the fixed point $i(e)$ has weight $w(e)$ and the fixed point $t(e)$ has weight $-w(e)$.

Now, let $\Gamma'$ be a multigraph obtained from $\Gamma$ by changing the direction of $e$ and changing the label $w(e)$ of $e$ by its negative $-w(e)$. Let $e'$ be the edge obtained from $e$ by reversing the direction of $e$. Then $w(e')=-w(e)$, $i(e')=t(e)$, and $t(e')=i(e)$. Since the label of $e'$ is $-w(e)$, $t(e')=i(e)$ has weight $-w'(e)=-(-w(e))=w(e)$ and $i(e')=t(e)$ has weight $w(e')=-w(e)$. Therefore, $\Gamma'$ also describes $M$.

Therefore, when a multigraph $\Gamma$ describes $M$, if we change the direction of an edge and the label of the edge by its negative simultaneously, the resulting multigraph also describes $M$.
\end{rem}

\begin{figure}
\centering
\begin{subfigure}[b][6.5cm][s]{.4\textwidth}
\centering
\begin{tikzpicture}[state/.style ={circle, draw}]
\node[state] (a) {$p_{0}$};
\node[state] (b) [above right=of a] {$p_1$};
\node[state] (c) [above=of b] {$p_2$};
\node[state] (d) [above left=of c] {$p_3$};
\path (a) [->] edge node[right] {$(1,0)$} (b);
\path (a) [->] edge node[left] {$(2,0)$} (c);
\path (a) [->] edge node[left] {$(0,1)$} (d);
\path (b) [->] edge node [right] {$(1,0)$} (c);
\path (b) [->] edge node [left] {$(-1,1)$} (d);
\path (c) [->] edge node [right] {$(-2,1)$} (d);
\end{tikzpicture}
\caption{Multigraph describing $\mathbb{CP}^3$ in Example \ref{e26}}\label{fig2-1}
\end{subfigure}
\begin{subfigure}[b][6.5cm][s]{.4\textwidth}
\centering
\vfill
\begin{tikzpicture}[state/.style ={circle, draw}]
\node[state] (a) {$p_1$};
\node[state] (b) [above right=of a] {$p_2$};
\node[state] (c) [above=of b] {$p_3$};
\node[state] (d) [above left=of c] {$p_4$};
\path (a) [->] [bend left =10] edge node[right] {$2$} (d);
\path (a) [->] [bend right =10] edge node [right] {$3$} (d);
\path (a) [->] edge node [right] {$1$} (b);
\path (b) [->] [bend left=20] edge node [right] {$1$} (c);
\path (b) [->] [bend right =20]  edge node [right] {$a$} (c);
\path (c) [->] edge node [right] {$1$} (d);
\end{tikzpicture}
\vfill
\caption{Multigraph describing Fano 3-fold in Example \ref{e27}}\label{fig2-2}
\end{subfigure}
\caption{Multigraphs for non-GKM manifolds}\label{fig2}
\end{figure}
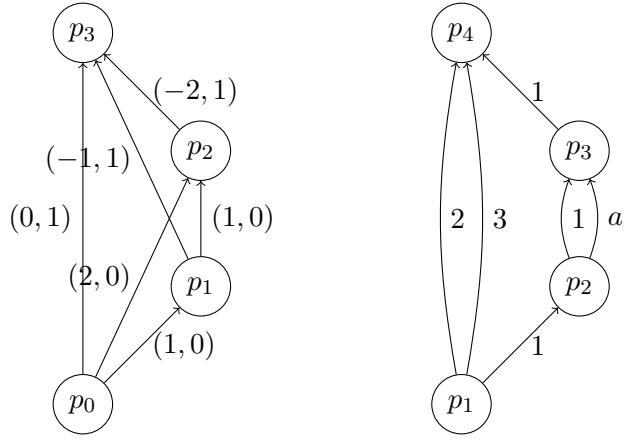

We give examples of torus actions on compact (almost) complex manifolds with isolated fixed points, which does not satisfy the GKM condition.

\begin{exa} \label{e26}
Let $T^2$ act on $\mathbb{CP}^3$ by
\begin{center}
$(g_1,g_2) \cdot [z_0:z_1:z_2:z_3]=[z_0:g_1 z_1: g_1^2 z_2: g_2 z_3]$ 
\end{center}
for all $g=(g_1,g_2) \in T^2$. The action has 4 fixed points, $p_0=[1:0:0:0]$, $p_1=[0:1:0:0]$, $p_2=[0:0:1:0]$, and $p_3=[0:0:0:1]$. 

Near $p_2$, using local coordinates $(\frac{z_0}{z_2},\frac{z_1}{z_2},\frac{z_3}{z_2})$, $T^2$ acts near $p_2$ by
\begin{center}
$(g_1,g_2) \cdot (\frac{z_0}{z_2},\frac{z_1}{z_2},\frac{z_3}{z_2})=(\frac{z_0}{g_1^2 z_2},\frac{g_1 z_1}{g_1^2 z_2},\frac{g_2 z_3}{g_1^2 z_2})=(g_1^{-1} \frac{z_0}{z_2},g_1^{-1} \frac{z_1}{z_2},g_1^{-2} g_2 \frac{z_3}{z_2})$
\end{center}
and hence the weights at $p_2$ are $\{(-2,0), (-1,0), (-2,1)\}$. Since the weights at $p_2$ are not pairwise linearly independent, $\mathbb{CP}^3$ with this action is not a GKM manifold; unlike GKM manifolds, a component of $M^{\ker (-1,0)}$ containing $p_2$ is not a 2-sphere (which is the case for a GKM manifold), but it is $\mathbb{CP}^2=\{[z_0:z_1:z_2:0]\}$.

The weights at the fixed points $p_0$, $p_1$, and $p_3$ are $\{(1,0), (2,0), (0,1)\}$, $\{(-1,0),(1,0),(-1,1)\}$, and $\{(0,-1),(1,-1),(2,-1)\}$, respectively. Figure \ref{fig2-1} describes $\mathbb{CP}^3$ for this action. 

On the other hand, we can give another $T^2$-action on $\mathbb{CP}^3$ so as to be a GKM manifold, for instance 
\begin{center}
$(g_1,g_2) \cdot [z_0:z_1:z_2:z_3]=[z_0:g_1^2 z_1: g_1^3 z_2: g_2 z_3]$.
\end{center}
\end{exa}

\begin{exa} \label{e27}
In \cite{A}, Ahara explicitly described circle actions on the Fano 3-folds $V_5$ and $V_{22}$ with 4 fixed points that have weights
\begin{center}
$\{1,2,3\}$, $\{-1,1,a\}$, $\{-1,-a,1\}$, $\{-1,-2,-3\}$,
\end{center}
where $a=4$ for $V_5$ and $a=5$ for $V_{22}$. Figure \ref{fig2-2} is a multigraph describing the manifold. McDuff provided compatible symplectic structures on these actions on $V_5$ and $V_{22}$ \cite{Mc}.
\end{exa}

Let the circle act symplectically on a compact symplectic manifold. In \cite{T}, Tolman proved that if $p$ and $p'$ are fixed points that are in the same connected component of $M^{\mathbb{Z}_k}$, then the $S^1$-weights at $p$ and at $p'$ are equal modulo $k$. The proof naturally holds for almost complex manifolds \cite{GS} and unitary manifolds \cite{Jan4}. We review the proof to extend the result to torus actions on unitary manifolds. Let a torus $T^k$ act on a compact unitary manifold $M$. Let $p$ and $p'$ be fixed points which lie in the same connected component $F$ of $M^{\ker w}$. Since $\ker w$ fixes $F$, the weights of the representation $\ker w$ on the tangent space $T_qM$ are the same for all $q \in F$. If $q \in F$ is fixed by the action of the torus $T^k$, the weights of the $\ker w$-action on $T_qM$ are the reduction modulo $w$ of the $T^k$-weights. Therefore, the $T^k$-weights at $p$ and $p'$ are equal modulo $w$, for $p, p' \in F$.

\begin{lemma} \label{l27}
Fix an element $w$ in $\mathbb{Z}^k$. Let a $k$-dimensional  torus $T^k$ act on a compact unitary manifold $M$. Let $p$ and $p'$ be fixed points which lie in the same component of $M^{\ker w}$. Then the $T^k$-weights at $p$ and at $p'$ are equal modulo $w$.
\end{lemma}

With the above, we show that a multigraph describing an almost complex torus manifold is a graph.

\begin{pro} \label{p28}
Let $M$ be an almost complex torus manifold. Then there is a graph describing $M$ that has no self-loops. In other words, there is a multigraph describing $M$ that has no multiple edges and no self-loops.
\end{pro}

\begin{proof}
By Proposition \ref{p22}, there exists a multigraph $\Gamma$ describing $M$ that has no self-loops. Suppose on the contrary that there are multiple edges $e_1$ and $e_2$ between two vertices $p$ and $q$. By changing the direction of $e_i$ and by changing the label $w(e_i)$ of $e_i$ by its negative $-w(e_i)$ simultaneously, we may assume that the edges $e_1$ and $e_2$ have initial vertex $p$ and terminal vertex $q$; see Remark \ref{r25}. Let $w_{p,1}$, $\cdots$, $w_{p,n}$ be the weights at $p$ and let $w_{q,1}$, $\cdots$, $w_{q,n}$ be the weights at $q$, where $\dim M=2n$. By permuting $w_{p,i}$ and by permuting $w_{q,i}$ if necessary, we may assume that $e_1$ corresponds to $w_{p,1}$ and $w_{q,1}$, and $e_2$ corresponds to $w_{p,2}$ and $w_{q,2}$, that is, $w(e_1)=w_{p,1}$, $-w(e_1)=w_{q,1}$, $w(e_2)=w_{p,2}$, and $-w(e_2)=w_{q,2}$. In other words, $w_{q,1}=-w_{p,1}$ and $w_{q,2}=-w_{p,2}$.

Since $\Gamma$ describes $M$, $p$ and $q$ are in the same component of $M^{\ker w_{p,1}}$. Therefore, by Lemma \ref{l27}, the weights at $p$ and at $q$ are equal modulo $w_{p,1}$. That is, there exists a bijection $\pi:\{1,\cdots,n\} \to \{1,\cdots,n\}$ such that $w_{p,i} \equiv w_{q, \pi(i)} \mod w_{p,1}$. 

We show that $\pi(1)=1$. For this, assume on the contrary that $\pi(1) \neq 1$. Then $w_{p,1} \equiv w_{q ,\pi(1)} \mod w_{p,1}$ and hence $w_{q, \pi(1)}$ is a multiple of $w_{p,1}$. Then $w_{q, \pi(1)}$ is a multiple of $w_{q,1}(=-w_{p,1})$, which leads to a contradiction since $\{w_{q,1},\cdots,w_{q,n}\}$ form a basis of $\mathbb{Z}^n$. Hence $\pi(1)=1$. 

Suppose that $\pi(2)=2$. Then we have $w_{p,2} \equiv w_{q,2} \mod w_{p,1}$. Since $w_{q,2}=-w_{p,2}$, this means that $w_{p,2} \equiv -w_{p,2} \mod w_{p,1}$ and hence $2w_{p,2}=a w_{p,1}$ for some integer $a$. This leads to a contradiction since $\{w_{p,1}, w_{p,2}=\frac{a w_{p,1}}{2} ,\cdots,w_{p,n}\}$ must form a basis of $\mathbb{Z}^n$.

Next, suppose that $\pi(2)>2$. Then we have $w_{p,2} \equiv w_{q,\pi(2)} \mod w_{p,1}$ and hence $w_{q,\pi(2)}=w_{p,2}+aw_{p,1}$ for some integer $a$. Then $w_{q,\pi(2)}=w_{p,2}+aw_{p,1}=-w_{q,2}-aw_{q,1}$ and hence $\{w_{q,1},w_{q,2},\cdots,w_{q,\pi(2)}=-w_{q,2}-aw_{q,1}, \cdots,w_{q,n}\}$ cannot form a basis of $\mathbb{Z}^n$. Therefore, $\Gamma$ has no multiple edges. \end{proof}

A geometric interpretation of Proposition \ref{p28} is as follows.

\begin{pro} \label{p29}
Let $M$ be a $2n$-dimensional almost complex torus manifold. Then for each fixed point $p$, there are $n$ 2-spheres on each of which the $T^n$-action on $M$ restricts to act, sharing $p$ as one fixed point in common with weight $w_{p,i}$, but no two of them share other fixed points. Here, $w_{p,1}$, $\cdots$, $w_{p,n}$ are the weights at $p$.
\end{pro}

\begin{proof}
By Proposition \ref{p28}, there is a graph $\Gamma$ describing $M$ that has no self-loops. Therefore, if we let $e_i$ be an edge of $\Gamma$ corresponding to the weight $w_{p,i}$ at $p$ and let $q_i$ be the other vertex of the edge $e_i$, then $q_1$, $\cdots$, $q_n$ are all distinct. Since $w_{p,1}$, $\cdots$, $w_{p,n}$ form a basis of $\mathbb{Z}^n$, for each $i \in \{1,\cdots,n\}$, a connected component $F_i$ of $M^{\ker w_{p,i}}$ containing $p$ and $q_i$ is a $2$-dimensional compact almost complex submanifold $M$. On $F_i$, the $T^n$-action on $M$ restricts to act, and this $T^n$-action on $F_i$ has fixed point $p$ and $q_i$ that have weight $w_{p,i}$ and $-w_{p,i}$. Therefore, $F_i$ is the 2-sphere. \end{proof}

For an almost complex torus manifold, since weights at each fixed point form a basis of $\mathbb{Z}^n$, they are pairwise linearly independent, so that it satisfies the GKM condition that; thus Corollary \ref{c24} and Proposition \ref{p28} imply the following.

\begin{cor}
Let $M$ be an almost complex torus manifold. Then the associated GKM graph has no multiple edges.
\end{cor}

\section{Hirzebruch genera of almost complex torus manifolds} \label{s3}

As for almost complex manifolds, for a compact unitary manifold its Hirzebruch $\chi_y$-genus is defined as the genus of the power series $\frac{x(1+ye^{-x(1+y)})}{1-e^{-x(1+y)}}$. Let the circle group $S^1$ act on a compact unitary manifold. For a fixed component $F$ of the action, let $d(-,F)$ ($d(+,F)$) be the number of negative (positive) weights of $F$, respectively. In \cite{Kos}, Kosniowski proved the following formula for the Hirzebruch $\chi_y$-genus of $M$, called the \textbf{Kosniowski formula}.

\begin{center}
$\displaystyle \chi_y(M)=\sum_{F \subset M^{S^1}} (-y)^{d(-,F)} \cdot \chi_y(F)=\sum_{F \subset M^{S^1}} (-y)^{d(+,F)} \cdot \chi_y(F)$.
\end{center}

Kosniowski first proved this formula for holomorphic vector fields on complex manifolds \cite{Kos}. Hattori and Taniguchi proved the same formula for unitary manifolds \cite{HT}. Kosniowski and Yahia later reproved the same result \cite{KY}. 

\begin{theo}[Kosniowski formula] \cite{HT}, \cite{KY} \label{t31}
Let the circle act on a compact unitary manifold $M$. For each connected component $F$ of the fixed point set $M^{S^1}$, let $d(-,F)$ and $d(+,F)$ be the numbers of negative weights and positive weights in the normal bundle $NF$ of $F$, respectively. Then the Hirzebruch $\chi_y$-genus $\chi_y(M)$ of $M$ satisfies
\begin{center}
$\displaystyle \chi_y(M)=\sum_{F \subset M^{S^1}} (-y)^{d(-,F)} \cdot \chi_y(F)=\sum_{F \subset M^{S^1}} (-y)^{d(+,F)} \cdot \chi_y(F)$.
\end{center}
\end{theo}

We first note that if a torus action on a $2n$-dimensional compact almost complex manifold $M$ has isolated fixed points, then the coefficients $a_i(M)$ of the Hirzebruch $\chi_y(M)=\sum_{i=0}^n a_i(M) \cdot (-y)^i$ are non-negative.

\begin{lem} \label{l32}
Let a $k$-dimensional torus $T^k$ act on a $2n$-dimensional compact almost complex manifold $M$ with isolated fixed points. Let $\chi_y(M)=\sum_{i=0}^n a_i(M) \cdot (-y)^i$ be the Hirzebruch $\chi_y$-genus of $M$. Then the following hold:
\begin{enumerate}
\item $a_i(M) \geq 0$ for $0 \leq i \leq n$.
\item The total number of fixed points is $\sum_{i=0}^n a_i(M)$.
\item For $0 \leq i \leq n$, $a_i(M)=a_{n-i}(M)$.
\end{enumerate}
\end{lem}

\begin{proof}
Because $M$ is compact, there is a finite number of orbit types. Therefore, there exists a subcircle $S^1$ of the torus $T^k$ that has the same fixed point set as the $T^k$-action on $M$, i.e., $M^{S^1}=M^{T^k}$. Since there are only isolated fixed points and the Hirzebruch $\chi_y$-genus of a point is $1$, by applying Theorem \ref{t31} to the $S^1$-action on $M$,
\begin{center}
$\displaystyle \chi_y(M)=\sum_{p \in M^{S^1}} (-y)^{d(-,p)} \cdot \chi_y(p)=\sum_{p \in M^{S^1}} (-y)^{d(-,p)} \cdot 1$.
\end{center}
Therefore, $a_i(M)$ is equal to the number of the $S^1$-fixed points ($T^k$-fixed points) that have exactly $i$ negative $S^1$-weights, and hence is non-negative. The second claim also follows.

Suppose that a fixed point $q$ has exactly $i$ negative weights for the $S^1$-action. Then in the first equation 
\begin{center}
$\displaystyle \chi_y(M)=\sum_{p \in M^{S^1}} (-y)^{d(-,p)} \cdot \chi_y(p)$
\end{center}
of Theorem \ref{t31}, the fixed point $q$ contributes the coefficient of $(-y)^i$ by $1$. Similarly, in the second equation
\begin{center}
$\displaystyle \chi_y(M)=\sum_{p \in M^{S^1}} (-y)^{d(+,p)} \cdot \chi_y(p)$
\end{center}
of Theorem \ref{t31}, the fixed point $q$ contributes the coefficient of $(-y)^{n-i}$ by $1$. This implies (3). \end{proof}

With the Kosniowski formula, we are ready to prove our main result of this section.

\begin{theo} \label{t32}
Let $M$ be a $2n$-dimensional almost complex torus manifold. Then $a_i(M) > 0$ for $0 \leq i \leq n$, where $\chi_y(M)=\sum_{i=0}^n a_i(M) \cdot (-y)^i$ is the Hirzebruch $\chi_y$-genus of $M$. In particular, the Todd genus of $M$ is positive.
\end{theo}

\begin{proof}
We prove by induction on the dimension $2n$ of the manifold $M$. The claim holds if $n=0$; if $M$ is a point, its Hirzebruch $\chi_y$-genus and Todd genus are both $1$.

Suppose that the theorem holds for all almost complex torus manifolds $N$ such that $\dim N<2n$. Let $F_0$ be a characteristic submanifold of $M$, that is, a connected component of real codimension 2 almost complex submanifold of $M$ fixed by some circle subgroup of $T^n$, which contains some $T^n$-fixed point $p \in M^{T^n}$.

Such a characteristic submanifold always exists and here is an example. Pick a fixed point $p \in M^{T^n}$. Let $w_{p,1}$, $\cdots$, $w_{p,n}$ be the $T^n$-weights at $p$. Since the $T^n$-action on $M$ is effective, $w_{p,1}$, $\cdots$, $w_{p,n}$ form a basis of $\mathbb{Z}^n$. Consider $U=\{x \in \mathfrak{t} \, | \, w_{p,2}(x)=\cdots=w_{p,n}(x)=0\} \subset \mathfrak{t}$, where $\mathfrak{t}$ denotes the Lie algebra of the torus $T^n$. Then $U$ is 1-dimensional. Let $\xi$ be a non-zero element in $U$ and let $S^1$ be a subcircle of $T^n$ generated by $\xi$. Then a component $F_0$ of the set $M^S$ of points in $M$ that are fixed by this $S^1$-action that contains $p$, is $(2n-2)$-dimensional.

Let $F$ be a component of $M^S$, not necessarily $F_0$. The $T^n$-action on $M$ restricts to act on $F$ and this $T^n$-action on $F$ has fixed points $F^{T^n}=M^{T^n} \cap F$, which are finite. Applying Lemma \ref{l32} to the $T^n$-action on $F$, it follows that $a_i(F) \ge 0$ for $0 \leq i \leq m$, where $\dim F=2m$ and $\chi_y(F)=\sum_{i=0}^m a_i(F) \cdot (-y)^i$ is the Hirzebruch $\chi_y$-genus of $F$.

Applying Theorem \ref{t31} to the $S^1$-action on $M$, we have
\begin{equation} \label{eq1}
\displaystyle \chi_y(M)=\sum_{F \subset M^{S^1}} (-y)^{d(-,F)} \cdot \chi_y(F),
\end{equation}
and
\begin{equation} \label{eq2}
\displaystyle \chi_y(M)=\sum_{F \subset M^{S^1}} (-y)^{d(+,F)} \cdot \chi_y(F).
\end{equation}

Let $k \in \{0,1,\cdots,n\}$. We consider the coefficient of $(-y)^k$-term of $\chi_y(M)$. In Equation \eqref{eq1}, a fixed component $F$ of $M^{S^1}$ contributes to $(-y)^k$-term of $\chi_y(M)$ by $a_i(F) \cdot (-y)^k$ if $i+d(-,F)=k$. Since $a_i(F) \geq 0$ for $0 \leq i \leq \frac{1}{2} \dim F$, in Equation \eqref{eq1}, $F$ contributes to the coefficient of $(-y)^k$-term of $\chi_y(M)$ by a non-negative integer. Similarly, in Equation \eqref{eq2}, $F$ contributes to the coefficient of $(-y)^k$-term of $\chi_y(M)$ by a non-negative integer $a_i(M)$ if $i+d(+,F)=k$.

The characteristic submanifold $F_0$ is a $(2n-2)$-dimensional compact connected almost complex submanifold, has an effective action of $T^{n-1}=T^n/S^1$ where $S^1$ is as above, and this $T^{n-1}$-action on $F_0$ has a fixed point $p$ since $p \in F_0^{T^{n-1}}=F_0 \cap M^{T^n}$. That is, $F_0$ is a $(2n-2)$-dimensional almost complex torus manifold. By inductive hypothesis, we have that $\chi^i(F_0) >0$ for $0 \leq i \leq n-1$. Since $\dim F_0=2n-2$, we have that either
\begin{enumerate}
\item $d(-,F_0)=0$ and $d(+,F)=1$, or
\item $d(-,F)=1$ and $d(+,F)=0$. 
\end{enumerate}

Suppose that $d(-,F_0)=0$ and $d(+,F_0)=1$. Since $d(-,F_0)=0$ and $a_i(F_0) >0$ for $0 \leq i \leq n-1$, for each $k \in \{0,1,\cdots,n-1\}$ the submanifold $F_0$ contributes the coefficient $a_k(M)$ of the $(-y)^k$-term in $\chi_y(M)$ by a positive integer in Equation \eqref{eq1}. Because any other fixed component contributes $a_k(M)$ by a non-negative integer, it follows that $a_k(M) >0$ for $k \in \{0,1,\cdots,n-1\}$. Similarly, since $d(+,F_0)=1$, $a_{n-1}(F_0)>0$ and $a_i(F) \geq 0$ for all fixed components $F \subset M^{S^1}$ and for all $i$, Equation \eqref{eq2} implies that $a_n(M) >0$. Therefore, $a_i(M)>0$ for $0 \leq i \leq n$. This proves the theorem when $d(-,F_0)=0$. The proof for the case that $d(-,F)=1$ and $d(+,F)=0$ is similar. 

The Todd genus $\mathrm{Todd}(M)$ of $M$ is the Hirzebruch $\chi_y$-genus $\chi_y(M)$ of $M$ evaluated at $y=0$, and we have $\mathrm{Todd}(M)=\chi_0(M)=\sum_{i=0}^n a_i(M) \cdot (-y)^i|_{y=0}=a_0(M)>0$. This proves the last claim of this theorem. \end{proof}

While Corollary \ref{c16} follows from Proposition \ref{p14}, it also follows from Lemma \ref{l32} and Theorem \ref{t32}. In the remaining of this section, we discuss that the conclusion of Theorem \ref{t15} is optimal in a sense that if we weaken any condition, the conclusion need not hold.

First, we cannot reduce the dimension of a torus acting on an almost complex manifold. As $G(2)/SU(3)$, the 6-sphere $S^6$ admits an almost complex structure, and admits an effective $T^2$-action with 2 fixed points. We can find a subcircle $S^1$ of $T^2$ that has the same fixed point set as the $T^2$-action on $S^6$, and the $S^1$-weights at the fixed points are $\{-a-b,a,b\}$ and $\{-a,-b,a+b\}$ for some positive integers $a$ and $b$. By Theorem \ref{t31}, the Hirzebruch $\chi_y$-genus of $S^6$ is $\chi_y(S^6)=-y+y^2$. In particular, the Todd genus of $S^6$ is zero; $\mathrm{Todd}(S^6)=\chi_0(S^6)=0$.

Second, we cannot weaken almost complex structure to unitary structure. To see this, consider the $2n$-dimensional sphere $S^{2n}=\{(z_1,\cdots,z_n,x) \in \mathbb{C}^n \times \mathbb{R} \, | \, x^2+\sum_{i=1}^n |z_i|^2=1\}$. Since $TS^{2n} \oplus \underline{\mathbb{R}}^2 = S^{2n} \times \mathbb{C}^{n+1}$, $S^{2n}$ admits a unitary structure. The Hirzebruch $\chi_y$-genus of $S^{2n}$ with this unitary structure vanishes; $\chi_y(TS^{2n} \oplus \underline{\mathbb{R}}^2)=0$. In particular, the Todd genus of $S^{2n}$ with this unitary structure also vanishes. Let $T^n$ act on $S^{2n}$ by
\begin{center}
$(g_1,\cdots,g_n) \cdot (z_1,\cdots,z_n,x)=(g_1 z_1, \cdots, g_n z_n,x)$
\end{center}
for all $(g_1,\cdots,g_n) \in T^n \subset \mathbb{C}^n$ and for all $(z_1,\cdots,z_n,x) \in S^{2n}$. This action has two fixed points $p=(0,\cdots,0,1)$ and $q=(0,\cdots,0,-1)$ that both have $T^n$-weights $i_1$, $i_2$, $\cdots$, $i_n$, where $i_j=(0,\cdots,0,1,0,\cdots,0)$ denote the standard unit vectors in $\mathbb{Z}^n$, $1 \leq j \leq n$. Let $S^1$ be a subcircle of $T^n$ generated by an element $(1,1,\cdots,1)$ in the Lie algebra $\mathfrak{t}$ of $T^n$. The $S^1$-weights at $p$ and $q$ are both $\{1,1,\cdots,1\}$, whereas $\chi_y(p)=+1$ and $\chi_y(q)=-1$. Thus $d(-,p)=d(-,q)=0$. By Theorem \ref{t31},
\begin{center}
$\chi_y(S^{2n})=(-y)^{d(-,p)}\cdot \chi_y(p)+(-y)^{d(-,q)}\cdot \chi_y(q)=+1-1=0$.
\end{center}
Therefore, Theorem \ref{t31} also confirms that $\chi_y(TS^{2n} \oplus \underline{\mathbb{R}}^2)=0$.

\section{Problem of Petrie type} \label{s4}

In \cite{Kob}, Kobayashi proved that for a torus action on a compact manifold, its Euler number is equal to the Euler number of its fixed point set.

\begin{theo} \cite{Kob} \label{t41}
Let a torus $T^k$ act on a compact manifold $M$. Then the Euler number $\chi(M)$ of $M$ is equal to the sum of Euler numbers of its fixed components. That is,
\begin{center}
$\displaystyle \chi(M)=\sum_{F \subset M^{T^k}} \chi(F)$.
\end{center}
\end{theo}

If an almost complex torus manifold has the same Euler number as $\mathbb{CP}^n$, then it has the same Hirzebruch $\chi_y$-genus as $\mathbb{CP}^n$.

\begin{theo} \label{t42}
Let $M$ be a $2n$-dimensional almost complex torus manifold with Euler number $n+1$. Then the Hirzebruch $\chi_y$-genus of $M$ is $\chi_y(M)=\sum_{i=0}^n (-y)^i$, the Todd genus of $M$ is $1$, and the signature $M$ is $\sum_{i=0}^n (-1)^i$.
\end{theo}

\begin{proof}
The manifold $M$ has finitely many fixed points. The Euler number of a point is $1$. Thus by Theorem \ref{t41} $M$ has precisely $n+1$ fixed points. By Theorem \ref{t15}, $a_i(M)>0$ for $0 \leq i \leq n$, where $\chi_y(M)=\sum_{i=0}^n a_i(M) \cdot (-y)^i$ is the Hirzebruch $\chi_y$-genus of $M$. By Lemma \ref{l32}, $\sum_{i=0}^n a_i(M)$ is equal to the total number of fixed points, which is $n+1$. It follows that $a_i(M)=1$ for $0 \leq i \leq n$, that is, $\chi_y(M)=\sum_{i=0}^n (-y)^i$. The other claims follow. \end{proof}

Let a $k$-dimensional torus $T^k$ act on a compact oriented manifold $M$. The projection map $\pi : M \times_{T^k} ET^k \rightarrow BT^k$ induces a push-forward map
\begin{center}
$\displaystyle \pi_{*} : H_{T^k}^{i}(M;\mathbb{Z}) \rightarrow H^{i-\dim M}(BT^k;\mathbb{Z})$
\end{center}
for all $i \in \mathbb{Z}$, where $BT^k$ is the classifying space of $T^k$. The projection map is given by integration over the fiber $M$, and is denoted by $\int_M$.

\begin{theo} \emph{[The Atiyah-Bott-Berline-Vergne localization theorem]} \cite{AB} \label{t43}
Let a torus $T^k$ act on a compact oriented manifold $M$. Let $ \alpha \in H_{T^k}^{\ast}(M;\mathbb{Q})$. As an element of $\mathbb{Q}(t)$,
\begin{center}
$\displaystyle \int_{M} \alpha = \sum_{F \subset M^{T^k}} \int_{F} \frac{\alpha|_{F}}{e_{T^k}(NF)}$,
\end{center}
where the sum is taken over all fixed components, and $e_{T^k}(NF)$ is the equivariant Euler class of the normal bundle to $F$ in $M$.
\end{theo}

In \cite{Jan2}, a circle action on an almost complex manifold with few fixed points is classified; also see \cite{Jan1} for symplectic actions.

\begin{theorem} \label{t44}\cite{Jan2} Let the circle act on a compact, connected almost complex manifold $M$. 
\begin{enumerate}
\item If there is exactly one fixed point, then $M$ is a point. 
\item If there are exactly two fixed points, then either $\dim M=2$ or $\dim M=6$. If $\dim M=2$, $M$ is the 2-sphere and the weights at the fixed points are $\{a\}$ and $\{-a\}$ for some some positive integer $a$. If $\dim M=6$, then the weights at the fixed points are $\{-a-b,a,b\}$ and $\{-a,-b,a+b\}$ for some positive integers $a$ and $b$. 
\item If there are exactly three fixed points, then $\dim M=4$. Moreover, the weights at the fixed points are $\{a+b,a\},\{-a,b\}$, and $\{-b,-a-b\}$ for some positive integers $a$ and $b$.
\end{enumerate}
\end{theorem}

We first prove Theorem \ref{t19} in dimension up to 4 separately.

\begin{theo} \label{t45}
Let $M$ be a $2n$-dimensional almost complex torus manifold with Euler number $n+1$. Suppose that $n \leq 2$. Then a graph describing $M$ and weights at the fixed points are the same as those of $\mathbb{CP}^n$, where $\mathbb{CP}^n$ is equipped with a linear $T^n$-action of Example \ref{e111}.
\end{theo}

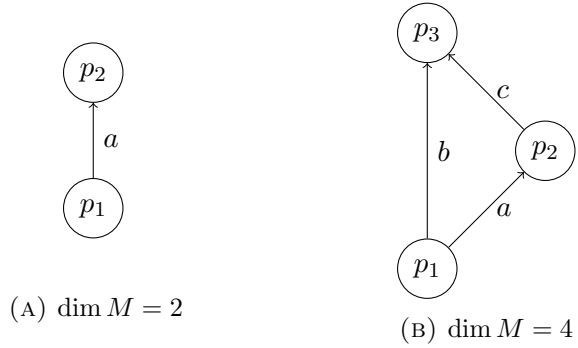
\begin{figure} 
\centering
\begin{subfigure}[b][4.3cm][s]{.4\textwidth}
\centering
\vfill
\begin{tikzpicture}[state/.style ={circle, draw}]
\node[state] (a) {$p_{1}$};
\node[state] (b) [above=of a] {$p_2$};
\path (a) [->] edge node[right] {$a$} (b);
\end{tikzpicture}
\vfill
\caption{$\dim M=2$}\label{fig3-1}
\end{subfigure}
\begin{subfigure}[b][4.3cm][s]{.4\textwidth}
\centering
\vfill
\begin{tikzpicture}[state/.style ={circle, draw}]
\node[state] (a) {$p_{1}$};
\node[state] (b) [above right=of a] {$p_2$};
\node[state] (c) [above left=of b] {$p_{3}$};
\path (a) [->] edge node[right] {$a$} (b);
\path (b) [->] edge node [right] {$c$} (c);
\path (a) [->] edge node [right] {$b$} (c);
\end{tikzpicture}
\vfill
\caption{$\dim M=4$}\label{fig3-2}
\end{subfigure}
\caption{Graphs describing $M$ in Theorem \ref{t45}}\label{fig3}
\end{figure}

\begin{proof}
If $n=0$, $M$ and $\mathbb{CP}^0$ are both points and the theorem holds. If $n=1$, the theorem follows from (2) of Theorem \ref{t44}; Figure \ref{fig3-1} describes $M$. Therefore, from now on, let $n=2$. By Proposition \ref{p28} (also see Remark \ref{r25}), Figure \ref{fig3-2} describes $M$, for some non-zero elements $a$, $b$, and $c$ in $\mathbb{Z}^2$. Then $p_1$ has weights $a$ and $b$, $p_2$ has weights $-a$ and $c$, and $p_3$ has weights $-b$ and $-c$.

We push-forward the equivariant cohomology class $1$ in the equivariant cohomology; for a dimensional reason that $\int_M$ is a map from $H_{T^2}^i(M;\mathbb{Z})$ to $H^{i - 4}(BT^2;\mathbb{Z})$, the image of $1$ under the map $\int_M$ is zero;
\begin{center}
$\displaystyle \int_M 1=0$.
\end{center}
Applying Theorem \ref{t43} to the $T^2$-action on $M$, taking $\alpha=1$,
\begin{center}
$\displaystyle \int_M 1=\sum_{p \in M^{T^k}} \frac{1}{e_{T^k}(N_pM)}=\sum_{p \in M^{T^k}} \frac{1}{e_{T^k}(T_pM)}=\sum_{p \in M^{T^2}} \frac{1}{\prod_{i=1}^2 w_{p,i}}=\frac{1}{ab}+\frac{1}{(-a)c}+\frac{1}{(-b)(-c)}$.
\end{center}
Therefore, $c=b-a$. Letting $c=b-a$, Figure \ref{fig3-2} also describes a linear $T^2$-action on $\mathbb{CP}^2$ given by
\begin{center}
$g \cdot [z_0:z_1:z_2]=[z_0:g^a z_1:g^b z_2]$,
\end{center}
see Example \ref{e111} for more details.
\end{proof}

With the above, we are ready to prove our main result of this section.

\begin{theo} \label{t46}
Let $M$ be a $2n$-dimensional almost complex torus manifold with Euler number $n+1$. Then the following invariants of $M$ and $\mathbb{CP}^n$ are equal. Here, $\mathbb{CP}^n$ is equipped with a linear $T^n$-action of Example \ref{e111}.
\begin{enumerate}
\item A (multi)graph describing it.
\item The weights at the fixed points.
\item All the Chern numbers.
\item Equivariant cobordism class.
\item The Hirzebruch $\chi_y$-genus.
\item The Todd genus.
\item The signature.
\end{enumerate}
If furthermore the action on $M$ is equivariantly formal, the following invariants of $M$ and $\mathbb{CP}^n$ are also equal.
\begin{enumerate}
\item[(8)] The rational equivariant cohomology.
\item[(9)] The Chern classes.
\end{enumerate}
\end{theo}

\begin{figure} 
\centering
\begin{tikzpicture}[state/.style ={circle, draw}]
\node[state] (a) {$p_{0}$};
\node[state] (b) [above right=of a] {$p_i$};
\node[state] (c) [above left=of b] {$p_j$};
\path (a) [->] edge node[right] {$w_{0,i}$} (b);
\path (b) [->] edge node [right] {$w_{i,j}$} (c);
\path (a) [->] edge node [right] {$w_{0,j}$} (c);
\end{tikzpicture}
\caption{Graph describing $T^n$-action on $F_0$}\label{fig4}
\end{figure}
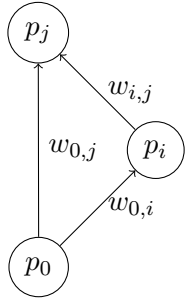

\begin{proof}
We prove (1) of the theorem; as in Remark \ref{r111}, all the others follow from (1). By Theorem \ref{t45}, (1) of the theorem holds if $\dim M \leq 4$. Therefore, from now on, we suppose that $\dim M>4$. 

Since there are only finitely many fixed points and the Euler number of a point is $1$, by Theorem \ref{t41}, $M$ has exactly $n+1$ fixed points. By Proposition \ref{p28}, there is a multigraph $\Gamma$ describing $M$ that has no multiple edges and no self-loops. We label the $n+1$ fixed points (vertices) by $p_0$, $p_1$, $\cdots$, $p_n$. Since there are $n+1$ vertices, each vertex has $n$ edges, and there are no multiple edges between any two vertices, for $i \neq j$, there is exactly one edge between $p_i$ and $p_j$. 

As in Remark \ref{r25}, by reversing the direction of an edge and the label of the edge by its negative simultaneously, if $e$ is an edge from $p_i$ to $p_j$, we may assume that $i < j$ (index increasing); the resulting multigraph also describes $M$, which we also denote by $\Gamma$. If an edge $e$ has initial vertex $p_i$ and terminal vertex $p_j$ with $i<j$, we denote the label of the edge $e$ by $w_{i,j}$. We will show that $w_{i,j}=w_{0,j}-w_{0,i}$ for $0<i<j$.

From now on we fix $0< i < j$. Let $T=\ker w_{0,i} \cap \ker w_{0,j}$; see the introduction for the definition of $\ker w_{0,i}$. Since the weights $w_{0,1}$, $\cdots$, $w_{0,n}$ at $p_0$ form a basis of $\mathbb{Z}^n$, $T$ is an $(n-2)$-dimensional subtorus of $T^n$. Let $F_0$ be a connected component of $M^T$ which contains $p_0$. Since $T$ only fixes $w_{0,i}$ and $w_{0,j}$ among the weights at $p_0$, it follows that $F_0$ is 4-dimensional and $F_0$ contains $p_0$, $p_i$, and $p_j$.

We claim that $F_0$ does not contain another fixed point. To prove this claim, assume on the contrary that $F_0$ contains another fixed point, say $p_k$, $k \neq 0,i,j$. The $T^n$-action on $M$ restricts to act on $F_0$, and the $T^n$-action on $F_0$ fixes $p_0$, $p_i$, $p_j$, and $p_k$. There can be more fixed points. Since the $T^n$-action on $F_0$ has at least 4 fixed points $p_0$, $p_i$, $p_j$, and $p_k$, applying Lemma \ref{l32} to the $T^n$-action on $F_0$, we have that $a_0(F_0)$, $a_1(F_0)$, $a_2(F_0) \geq 0$ and $a_0(F_0)+a_1(F_0)+a_2(F_0) \geq 4$, where $\chi_y(F_0)=a_0(F_0)-a_1(F_0)y+a_2(F_0)y^2$. In particular, at least one of $a_0(F_0)$, $a_1(F_0)$, and $a_2(F_0)$ is bigger than or equal to 2.

Now let $S$ be a subcircle of $T$ that has the same fixed point set as the $T$-action on $M$, i.e., $M^S=M^T$. Such an $S$ exists because $M$ is compact and there is a finite number of orbit types. 

Let $F$ be a fixed component of $M^S$. On $F$, the $T^n$-action on $M$ restricts to act, and its fixed point set is $F^{T^n}=M^{T^n} \cap F$, which consists of a finite number of fixed points. Applying Lemma \ref{l32} to the $T^n$-action on $F$, it follows that $a_m(F) \geq 0$ for all $0 \leq m \leq \frac{1}{2} \dim F$, where $\chi_y(F)=\sum_{m=0}^{\frac{1}{2} \dim F} a_m(F) \cdot (-y)^m$.

Applying Theorem \ref{t31} to the $S$-action on $M$, we have
\begin{center}
$\displaystyle \chi_y(M) = \sum_{F \subset M^S} (-y)^{d(-,F)} \cdot \chi_y(F)$,
\end{center}
where for each fixed component $F \subset M^S=M^T$, $d(-,F)$ is the number of negative $S$-weights of $F$. We have that $a_m(F) \geq 0$ for all $0 \leq m \leq \frac{1}{2} \dim F$, for all fixed components $F \subset M^S$. On the other hand, for the fixed component $F_0$ of $M^S=M^T$, at least one of $a_0(F_0)$, $a_1(F_0)$, and $a_2(F_0)$ is bigger than or equal to 2. Therefore, the coefficient of some $(-y)^m$-term of $\chi_y(M)$ must be bigger than or equal to 2. Since the Hirzebruch $\chi_y$-genus of $M$ is $\chi_y(M)=\sum_{m=0}^n (-y)^n$, we get a contradiction. Therefore, the only fixed points of the $T^n$-action on $F_0$ are $p_0$, $p_i$, and $p_j$.

The $T^n$-action on $M$ restricts to act on the 4-dimensional compact almost complex submanifold $F_0$ and this $T^n$-action on $F_0$ has fixed points $p_0$, $p_i$, and $p_j$ that have $T^n$-weights $\{w_{0,i},w_{0,j}\}$, $\{-w_{0,i}, w_{i,j}\}$, and $\{-w_{0,j},-w_{i,j}\}$, respectively (see Figure \ref{fig4}). For this $T^n$-action on $F_0$, we push-forward the equivariant cohomology class $1$ in the equivariant cohomology; for a dimensional reason ($\int_{F_0}:H_{T^n}^\ast(F_0) \to H^{\ast - \dim F_0}(BT^n)$), the image under $\int_{F_0}$ of $1$ is zero;
\begin{center}
$\displaystyle \int_{F_0} 1=0$.
\end{center}
Applying Theorem \ref{t43} to the $T^n$-action on $F_0$, taking $\alpha=1$,
\begin{center}
$\displaystyle \int_{F_0} 1=\sum_{p \in F_0^{T^n}} \frac{1}{e_{T^n}(N_pF_0)}=\sum_{p \in F_0^{T^n}} \frac{1}{e_{T^n}(T_pF_0)}=\frac{1}{w_{0,i}w_{0,j}}+\frac{1}{(-w_{0,i})w_{i,j}}+\frac{1}{(-w_{0,j})(-w_{i,j})}$.
\end{center}
Therefore, $w_{i,j}=w_{0,j}-w_{0,i}$. It follows that the graph $\Gamma$ describing $M$ agrees with the graph describing a linear $T^n$-action on $\mathbb{CP}^n$ of Example \ref{e111}, which proves (1). This completes the proof of the theorem. \end{proof}

In the remaining of this section, we discuss some interesting phenomenon that relates almost complex torus manifolds with minimal Euler number and simplices in Euclidean spaces.

Let $M$ be a $2n$-dimensional almost complex torus manifold with minimal Euler number, which is $n+1$. Then $M$ has $n+1$ fixed points, say $p_0$, $p_1$, $\cdots$, $p_n$. Proposition \ref{p28} tells us that there is a graph $\Gamma$ describing $M$. As in the proof of Theorem \ref{t46}, we may assume that any edge goes from $p_i$ to $p_j$ where $i<j$. By Proposition \ref{p29} each edge corresponds to the 2-sphere.

Now, we place $p_0$ at the origin in $\mathbb{R}^n$. Next, for each $1 \leq i \leq n$, since $w_{0,i}$ is the label of the edge from $p_0$ to $p_i$, we may place $p_i$ at the position of $w_{0,i}$ in $\mathbb{R}^n$, regarding $w_{0,i}$ as a vector in $\mathbb{R}^n$. Then, in theory, there is no need to hold that $w_{i,j}$, the label of the edge from $p_i$ to $p_j$, is precisely the difference vector $p_j-p_i$, but it does hold as we have seen in the proof of Theorem \ref{t46} that $w_{i,j}=w_{0,j}-w_{0,i}=p_j-p_i$.

Therefore, we can think of a graph describing an almost complex torus manifold $M$ with Euler number $n+1$ (which also describes the linear action on $\mathbb{CP}^n$ in Example \ref{e111}) as a simplex in $\mathbb{R}^n$ whose vertices $p_0$, $p_1$, $\cdots$, $p_n$ satisfy that $p_1-p_0$, $p_2-p_0$, $\cdots$, $p_n-p_0$ form a basis of $\mathbb{Z}^n$. In other words, an almost complex torus manifold with minimal Euler number is somehow classified by such a simplex.

\end{document}